\documentclass[12pt]{amsart}

\usepackage{amsfonts,amssymb,amsthm,mathtools}
\usepackage{aliascnt}
\usepackage[hidelinks]{hyperref}
\usepackage[nameinlink,noabbrev]{cleveref}

\theoremstyle{plain}
\newtheorem{theorem}{Theorem}[section]
\newaliascnt{proposition}{theorem}
\newtheorem{proposition}[proposition]{Proposition}
\aliascntresetthe{proposition}
\newaliascnt{lemma}{theorem}
\newtheorem{lemma}[lemma]{Lemma}
\aliascntresetthe{lemma}
\newaliascnt{corollary}{theorem}
\newtheorem{corollary}[corollary]{Corollary}
\aliascntresetthe{corollary}

\theoremstyle{definition}
\newaliascnt{definition}{theorem}
\newtheorem{definition}[definition]{Definition}
\aliascntresetthe{definition}

\theoremstyle{remark}
\newtheorem*{remark}{Remark}

\crefname{theorem}{theorem}{theorems}
\Crefname{theorem}{Theorem}{Theorems}
\crefname{proposition}{proposition}{propositions}
\Crefname{proposition}{Proposition}{Propositions}
\crefname{lemma}{lemma}{lemmas}
\Crefname{lemma}{Lemma}{Lemmas}
\crefname{corollary}{corollary}{corollaries}
\Crefname{corollary}{Corollary}{Corollaries}
\crefname{definition}{definition}{definitions}
\Crefname{definition}{Definition}{Definitions}

\newcommand{\defi}[1]{\textbf{#1}}
\newcommand{\bbZ}{\mathbb{Z}}
\newcommand{\calP}{\mathcal{P}}
\newcommand{\parent}[1]{\left( #1 \right)}

\hypersetup{hypertexnames=false}
\makeatletter
\providecommand*{\toclevel@theorem}{2}
\providecommand*{\toclevel@definition}{2}
\providecommand*{\toclevel@lemma}{2}
\providecommand*{\toclevel@corollary}{2}
\providecommand*{\toclevel@proposition}{2}
\providecommand*{\toclevel@remark}{2}
\makeatother

\newcommand{\primorial}[1]{\calP\parent{#1}}
\newcommand{\ksigma}[1]{\sigma^{[#1]}}
\newcommand{\ksigmagen}[2]{\sigma^{[#1]}_{#2}}
\newcommand{\rksigma}[1]{\ksigmagen{#1}{-1}}

\newcommand{\rkrho}[1]{\hat{\rho}^{[#1]}}
\newcommand{\RobinScale}[1]{e^\gamma\log\log #1}
\newcommand{\RobinScalePower}[2]{\parent{\RobinScale{#1}}^{#2}}
\newcommand{\rkRobinRatio}[2]{\frac{\rkrho{#1}(#2)}{\RobinScale{#2}}}
\newcommand{\MertensProduct}[1]{\prod_{p \leq #1}(1 - p^{-1})^{-1}}

\newcommand{\MertensThreshold}{19{,}421}
\newcommand{\PrimorialLogThreshold}{19{,}183}
\newcommand{\MertensCorrection}{0.10952}
\newcommand{\TransferCorrection}{0.42}

\newcommand{\CertifiedKappa}{100{,}000}
\newcommand{\CertifiedThetaArgument}{75{,}002}
\newcommand{\CertifiedThetaLower}{74670.66037438614}
\newcommand{\CertifiedThetaUpper}{74670.66037438615}
\newcommand{\CertifiedAnalyticCutoff}{246{,}590}
\newcommand{\CertifiedLastBadPrime}{227{,}951}
\newcommand{\CertifiedTailPrime}{227{,}977}
\newcommand{\CertifiedEndpointEvents}{39{,}454}
\newcommand{\CertifiedFirstEndpoint}{55{,}440}
\newcommand{\CertifiedEndpointMargin}{0.01151623688}
\newcommand{\CertifiedInitialChecks}{50{,}399}
\newcommand{\CertifiedDirectMinimumN}{10{,}080}
\newcommand{\CertifiedDirectMargin}{0.05748622492}

\begin{document}
\title[Reciprocal LCM powers and Robin's criterion]{Reciprocal LCM powers and Robin's criterion}
\author{Grant Molnar}
\email{molnar.grant.5772@gmail.com}

\begin{abstract}
For $\kappa > 1$, we study the reciprocal LCM-power function
\[
\sigma^{[\kappa]}_{-1}(n)
\coloneqq
\sum_{d \mid n}\mu(d)\,\sigma_{-1}(n / d)^\kappa.
\]
We determine its average and extremal orders.  A natural Robin-scale normalization is strictly increasing in $\kappa$ and tends to $\sigma_{-1}$, yielding an RH-equivalent criterion.  Unconditionally, the corresponding Robin inequality holds eventually on every bounded parameter interval.  Directed interval arithmetic proves the inequality for each $n > 5040$ and $1 < \kappa \le 100{,}000$.
\end{abstract}


\maketitle


\section{Introduction}

For a positive integer $n$, we write
\begin{equation}
\sigma_{-1}(n) \coloneqq \sum_{d \mid n}\frac1d = \frac{\sigma(n)}n.
\end{equation}
Gr\"{o}nwall's theorem gives the maximal-order formula \cite{Gronwall}
\begin{equation}
\limsup_{n \to \infty}\frac{\sigma_{-1}(n)}{\log\log n} = e^\gamma.
\end{equation}
Robin found the corresponding sharp explicit boundary.

\begin{theorem}[{Robin's theorem \cite[Th\'eor\`emes~1 et~2]{Robin}}]\label{thm:classical-robin}
The Riemann hypothesis is equivalent to
\begin{equation}\label{eq:classical-robin}
\sigma_{-1}(n) < \RobinScale{n}
\qquad\text{for }n > 5040.
\end{equation}
Unconditionally, we have
\begin{equation}\label{eq:robin-unconditional-bound}
\sigma_{-1}(n) < \RobinScale{n} + \frac{0.6483}{\log\log n}
\qquad\text{for }n \ge 3.
\end{equation}
\end{theorem}

We write $\mu$ for the M\"{o}bius function and $\omega(n)$ for the number of
distinct prime divisors of $n$.  We study a reciprocal LCM-power family that
interpolates back to $\sigma_{-1}$.

\begin{definition}\label{def:reciprocal-lcm-power}
For real $\kappa > 1$, the \defi{reciprocal LCM-power function} is
\begin{equation}\label{eqn:reciprocal-lcm-power}
\rksigma{\kappa}(n)
\coloneqq
\sum_{d \mid n}\mu(d)\sigma_{-1}(n / d)^\kappa.
\end{equation}
\end{definition}

The function $\rksigma{\kappa}$ is multiplicative, and its prime-power
values are
\begin{equation}\label{eq:rksigma-at-prime-powers}
\rksigma{\kappa}(p^a)
= \sigma_{-1}(p^a)^\kappa - \sigma_{-1}(p^{a - 1})^\kappa.
\end{equation}
M\"{o}bius inversion gives
\begin{equation}
\sigma_{-1}(n)^\kappa = \sum_{d \mid n}\rksigma{\kappa}(d).
\end{equation}
When $\kappa = k$ is a positive integer, \eqref{eqn:reciprocal-lcm-power} has the
LCM-convolution interpretation
\begin{equation}\label{eq:rksigma-lcm-identity}
\rksigma{k}(n)
= \sum_{[d_1, \ldots, d_k] = n}\frac1{d_1\cdots d_k}.
\end{equation}

\begin{theorem}[{\Cref{thm:reciprocal-average-order,thm:lower-order,thm:gronwall}}]\label{thm:intro-order-theory}
For each real $\kappa > 1$, we have
\[
    \sum_{n \le x}\rksigma{\kappa}(n)
    = \frac{(\log x)^\kappa}{\Gamma(\kappa + 1)}
    + O_\kappa\!\left((\log x)^{\kappa - 1}\right)
\]
as $x \to \infty$, and the summatory function admits a complete Selberg--Delange expansion.  Moreover, we have
\[
    \liminf_{n \to \infty}
    \frac{n\,\rksigma{\kappa}(n)}{\kappa^{\omega(n)}} = 1
\]
and
\[
    \limsup_{n \to \infty}
    \frac{n\,\rksigma{\kappa}(n)}
    {\kappa^{\omega(n)}\RobinScalePower{n}{\kappa - 1}} = 1.
\]
\end{theorem}

The Robin problem calls for one further normalization.  Taking the
$(\kappa - 1)$st root places the reciprocal coefficient on the same scale as
$\sigma_{-1}$.

\begin{definition}\label{def:normalized-reciprocal-lcm-power}
For real $\kappa > 1$, define
\begin{equation}\label{eq:normalized-function}
\rkrho{\kappa}(n)
\coloneqq
\left(\frac{n\,\rksigma{\kappa}(n)}{\kappa^{\omega(n)}}\right)^{1 / (\kappa - 1)}.
\end{equation}
\end{definition}

At each prime power, division by $\sigma_{-1}(p^a)$ turns the local factor in
\eqref{eq:normalized-function} into a power mean.  This orders the family in
$\kappa$ and recovers Robin's classical boundary in the limit.

\begin{theorem}[{\Cref{lem:power-mean} and \Cref{thm:reciprocal-robin-criterion}}]\label{thm:intro-robin-criterion}
For each integer $n > 1$, the function
$\kappa \mapsto \rkrho{\kappa}(n)$ is strictly increasing and
\[
    \lim_{\kappa \to \infty}\rkrho{\kappa}(n) = \sigma_{-1}(n).
\]
Consequently, for any fixed unbounded set $\mathcal K \subset (1, \infty)$,
the Riemann hypothesis is equivalent to
\[
    \rkrho{\kappa}(n) < \RobinScale{n}
\]
for each integer $n > 5040$ and every $\kappa \in \mathcal K$.
\end{theorem}

At fixed parameter, the reciprocal family behaves differently.  Every bounded
parameter range is eventually good unconditionally, and a finite endpoint
verification extends this to a complete interval of parameters.

\begin{theorem}[{\Cref{thm:bounded-eventual,thm:unconditional-interval}}]\label{thm:intro-unconditional}
Consider the reciprocal Robin inequality
\begin{equation}\label{eq:intro-reciprocal-robin}
    \rkrho{\kappa}(n) < \RobinScale{n}.
\end{equation}
For every real $\kappa_0 > 1$, there is an explicit finite cutoff $N_{\kappa_0}$ such that
\eqref{eq:intro-reciprocal-robin} holds for each integer $n \ge N_{\kappa_0}$ and every
real $1 < \kappa \le \kappa_0$.  Moreover, \eqref{eq:intro-reciprocal-robin} holds
unconditionally for each integer $n > 5040$ and every real
$1 < \kappa \le \CertifiedKappa$.
\end{theorem}

At $\kappa = \CertifiedKappa$, directed interval arithmetic verifies the finite
normalized endpoint prefix left by \Cref{thm:bounded-eventual} and the integers
preceding its first endpoint.  Certification and reproducibility details appear at the end of the paper.

The proof of \Cref{thm:intro-order-theory} begins with
\eqref{eq:rksigma-at-prime-powers}, which turns the local analysis into a
first-difference problem for $t^\kappa$.  The Dirichlet series of
$\rksigma{\kappa}$ then places its summatory function within the
Selberg--Delange method.

For \Cref{thm:intro-robin-criterion}, \eqref{eq:local-power-mean-representation} gives strict
parameter monotonicity and the limit as $\kappa \to \infty$.  To prove
\Cref{thm:intro-unconditional}, we instead extremize
$\rkrho{\kappa_0}(n)n^{-\epsilon}$ directly and transfer the resulting
endpoint inequalities between consecutive canonical endpoints.  The body first
proves the effective eventual theorem for each fixed $\kappa_0$, then certifies
the remaining finite prefix at $\kappa_0 = \CertifiedKappa$.

\subsection*{Prior work}

The LCM-power family of \cite{Fan-Mits-Molnar} is defined by
\begin{equation}
\ksigma{\kappa}(n) \coloneqq \sum_{d \mid n}\mu(n / d)\sigma(d)^\kappa.
\end{equation}
Its mean value is studied in \cite[Theorem~1.3]{Fan-Mits-Molnar}, and
\cite[Theorem~1.5]{Fan-Mits-Molnar} proves that, for each fixed real
$\kappa > 3 / 2$, the corresponding Robin-type inequality for all sufficiently
large integers is equivalent to the Riemann hypothesis.  The case $\kappa = 2$
also produces the reciprocal average appearing in
\cite[equation~(17)]{Fan-Mits-Molnar}.  The classical LCM convolution and the
exact $k$-fold convolute are reviewed in \cite[\S\S4.4 and~6]{TothSurvey}.

A related literature studies multivariable averages of functions of the least
common multiple over boxes or hyperbolic regions; see
\cite{EssouabriSalinasToth,HeymanToth,HilberdinkLucaToth,HilberdinkToth}.
For $k \ge 3$, Kim studies the reciprocal box sum
\[
S_k(x) = \sum_{d_1, \ldots, d_k \le x}\frac1{[d_1, \ldots, d_k]},
\]
and proves that its polynomial main term has degree $2^k - 1$
\cite[Theorem~1.1]{KimReciprocalLCM}.

Extremal sequences adapted to Gr\"onwall's quotient have also been studied in
the classical setting.  In particular, Caveney, Nicolas, and Sondow introduce
and study the GA numbers in \cite{CaveneyNicolasSondow}.

\section{Order theory}\label{sec:order-theory}

The local problem becomes a first-difference estimate for $t^\kappa$ through
\eqref{eq:rksigma-at-prime-powers}.  The resulting bounds multiply over the prime powers
dividing $n$ and lead successively to the lower, average, and maximal
orders of the reciprocal family.

\subsection{Local bounds and lower order}

By \eqref{eq:rksigma-at-prime-powers}, the local factor is a discrete derivative of $t^\kappa$.  The mean value theorem supplies the local estimates from which the lower order follows.

\begin{lemma}\label{lem:discrete-mvt}
For all $\kappa > 1$, each prime $p$, and each integer $a \geq 1$, we have
\begin{equation}\label{eq:discrete-mvt}
\frac{\kappa}{p^a}\, \sigma_{-1}(p^{a - 1})^{\kappa - 1} \; \leq \; \rksigma{\kappa}(p^a) \; \leq \; \frac{\kappa}{p^a}\, \sigma_{-1}(p^a)^{\kappa - 1}.
\end{equation}
\end{lemma}

\begin{proof}
The function $t \mapsto t^\kappa$ is convex on $[0, \infty)$, and its derivative $\kappa t^{\kappa - 1}$ is non-decreasing.  We apply the mean value theorem on the interval $[\sigma_{-1}(p^{a - 1}), \sigma_{-1}(p^a)]$.  There is a number $\xi$ in this interval such that
\[
\sigma_{-1}(p^a)^\kappa - \sigma_{-1}(p^{a - 1})^\kappa \; = \; \kappa\, \xi^{\kappa - 1}\, \bigl(\sigma_{-1}(p^a) - \sigma_{-1}(p^{a - 1})\bigr).
\]
The difference factor on the right is $p^{-a}$.  Bounding $\xi$ by the endpoints of $[\sigma_{-1}(p^{a - 1}), \sigma_{-1}(p^a)]$ now gives \eqref{eq:discrete-mvt}.
\end{proof}

Multiplying the upper bound in \eqref{eq:discrete-mvt} over all prime power divisors of $n$ gives a clean global upper bound.

\begin{corollary}\label{cor:global-upper-bound}
For each integer $n \geq 1$, we have
\begin{equation}\label{eq:global-upper-bound}
\rksigma{\kappa}(n) \; \leq \; \frac{\kappa^{\omega(n)}}{n}\, \sigma_{-1}(n)^{\kappa - 1}.
\end{equation}
\end{corollary}

\begin{proof}
For $n = 1$, the assertion is the identity $\rksigma{\kappa}(1) = 1$.  For $n > 1$, we multiply the right-hand side of \eqref{eq:discrete-mvt} over $p^a \parallel n$.  The three products are $\kappa^{\omega(n)}$, $n^{-1}$, and $\sigma_{-1}(n)^{\kappa - 1}$, respectively.
\end{proof}

A symmetric argument using the lower bound in \eqref{eq:discrete-mvt} produces a global lower bound with $\sigma_{-1}(p^{a - 1})$ in place of $\sigma_{-1}(p^a)$.

\begin{corollary}\label{cor:global-lower-bound}
For each integer $n > 1$, we have
\begin{equation}\label{eq:global-lower-bound}
\rksigma{\kappa}(n) \; \geq \; \frac{\kappa^{\omega(n)}}{n} \prod_{p^a \parallel n} \sigma_{-1}(p^{a - 1})^{\kappa - 1}.
\end{equation}
\end{corollary}

\begin{proof}
We multiply the left-hand side of \eqref{eq:discrete-mvt} over the
prime powers $p^a \parallel n$.  The factors $p^{-a}$ multiply to
$n^{-1}$, and the remaining local factors give \eqref{eq:global-lower-bound}.
\end{proof}

The two global inequalities meet asymptotically along the primes.  This
determines the lower order.

\begin{theorem}[{\Cref{thm:intro-order-theory}}]\label{thm:lower-order}
For every real $\kappa > 1$, we have
\[
\liminf_{n \to \infty}
\frac{n\,\rksigma{\kappa}(n)}{\kappa^{\omega(n)}} = 1.
\]
Equivalently, the lower order of $\rkrho{\kappa}(n)$ is $1$.
\end{theorem}

\begin{proof}
By \Cref{cor:global-lower-bound}, we have
\[
\frac{n\,\rksigma{\kappa}(n)}{\kappa^{\omega(n)}}
\ge
\prod_{p^a \parallel n}\sigma_{-1}(p^{a - 1})^{\kappa - 1}
\ge 1.
\]
For a prime $p$, we have
\begin{equation}\label{eq:prime-lower-order-limit}
\frac{p\,\rksigma{\kappa}(p)}{\kappa}
=
\frac p\kappa\left(\left(1 + \frac1p\right)^\kappa - 1\right)
\longrightarrow 1
\qquad\text{as }p \to \infty.
\end{equation}
Combining \Cref{cor:global-lower-bound} with \eqref{eq:prime-lower-order-limit} proves \Cref{thm:lower-order}.
\end{proof}

Combining \eqref{eq:robin-unconditional-bound} with \eqref{eq:global-upper-bound} gives a pointwise bound for the reciprocal family.

\begin{corollary}\label{cor:robin-derived-upper-bound}
For every real $\kappa > 1$ and each integer $n \geq 3$, we have
\[
\rksigma{\kappa}(n)
< \frac{\kappa^{\omega(n)}}n
\left(\RobinScale{n} + \frac{0.6483}{\log\log n}\right)^{\kappa - 1}.
\]
\end{corollary}

\begin{proof}
Combine \Cref{cor:global-upper-bound} with \eqref{eq:robin-unconditional-bound}.
\end{proof}

\subsection{Average order and reciprocal mean}

We next determine the average order of $\rksigma{\kappa}(n)$.  We put
\[
b_\kappa(n) \coloneqq n\,\rksigma{\kappa}(n).
\]
The Dirichlet series of $b_\kappa$ has singular factor
$\zeta(s)^\kappa$ at $s = 1$.  Selberg--Delange analysis therefore gives
\eqref{eq:weighted-reciprocal-full-expansion}.  Abel summation transfers \eqref{eq:weighted-reciprocal-full-expansion}
to \eqref{eq:reciprocal-full-expansion}, and the shifted Dirichlet
series determines the integration constant.  We then evaluate the series
at $s = 1$ to recover the convergent reciprocal mean used in the earlier
paper.

For $\kappa = 2$, a prior paper proved \cite[equation~(17)]{Fan-Mits-Molnar} that
\begin{equation}\label{eq:fan-mits-molnar-kappa-two-average}
\sum_{n \le x}\rksigma{2}(n)
= \frac12(\log x)^2 + O(\log x)
\end{equation}
as $x \to \infty$.
Our argument immediately below extends \eqref{eq:fan-mits-molnar-kappa-two-average} to every real $\kappa > 1$ and supplies all lower-order
terms.  \Cref{lem:average-order-euler-product} follows directly from \Cref{lem:discrete-mvt}.

\begin{lemma}\label{lem:average-order-euler-product}
For every real $\kappa > 1$, the Dirichlet series
\[
\mathcal B_\kappa(s)
\coloneqq
\sum_{n \ge 1}\frac{b_\kappa(n)}{n^s}
\]
converges absolutely for $\operatorname{Re}(s) > 1$.  There is a function $H_\kappa$ holomorphic on $\operatorname{Re}(s) > 1 / 2$.  It is bounded on every closed half-plane
$\operatorname{Re}(s) \geq 1 / 2 + \delta$, where $\delta > 0$, and satisfies
\begin{equation}\label{eq:average-order-factorization}
\mathcal B_\kappa(s) = \zeta(s)^\kappa H_\kappa(s)
\qquad\text{for }\operatorname{Re}(s) > 1,
\end{equation}
and
\begin{equation}\label{eq:h-at-one}
H_\kappa(1) = 1.
\end{equation}
Equivalently, for $\operatorname{Re}(s) > 0$, we have
\begin{equation}
\mathcal R_\kappa(s) = \zeta(s + 1)^\kappa H_\kappa(s + 1),
\end{equation}
where $H_\kappa(s + 1)$ is holomorphic on $\operatorname{Re}(s) > -1 / 2$.
\end{lemma}

\begin{proof}
The function $b_\kappa$ is multiplicative.  By \Cref{lem:discrete-mvt}, we have
\begin{equation}\label{eq:b-prime-power-bound}
0 < b_\kappa(p^a)
= p^a\rksigma{\kappa}(p^a)
\leq
\kappa\left(\frac p{p - 1}\right)^{\!\kappa - 1}
\leq C_\kappa
\end{equation}
for each prime $p$ and every $a \ge 1$.  Moreover, Taylor expansion at $1$ gives
\begin{equation}\label{eq:b-at-prime-expansion}
b_\kappa(p)
= p\left(\left(1 + \frac1p\right)^\kappa - 1\right)
= \kappa + O_\kappa(p^{-1})
\end{equation}
as $p \to \infty$ through the primes.
Thus the Euler product for $\mathcal B_\kappa(s)$ converges absolutely when $\operatorname{Re}(s) > 1$.

We write
\[
\mathcal B_{\kappa, p}(s)
\coloneqq
1 + \sum_{a \ge 1}\frac{b_\kappa(p^a)}{p^{as}}.
\]
Combining \eqref{eq:b-prime-power-bound} and \eqref{eq:b-at-prime-expansion}, as $p \to \infty$ through the primes and uniformly on every closed half-plane $\operatorname{Re}(s) \ge 1 / 2 + \delta$ with $\delta > 0$, gives
\begin{equation}\label{eq:local-b-expansion}
\mathcal B_{\kappa, p}(s)
= 1 + \kappa p^{-s}
+ O_{\kappa, \delta}\!\left(p^{-\operatorname{Re}(s) - 1} + p^{-2\operatorname{Re}(s)}\right).
\end{equation}
We define
\[
H_{\kappa, p}(s)
\coloneqq
(1 - p^{-s})^\kappa\mathcal B_{\kappa, p}(s),
\]
where the power is defined by the analytic logarithm on $|p^{-s}| < 1$.  Expanding the first factor and using \eqref{eq:local-b-expansion}, we obtain, as $p \to \infty$ through the primes and uniformly on the same half-plane,
\[
H_{\kappa, p}(s)
= 1 + O_{\kappa, \delta}\!\left(p^{-\operatorname{Re}(s) - 1} + p^{-2\operatorname{Re}(s)}\right).
\]
The error terms are summable over the primes uniformly on every closed half-plane
$\operatorname{Re}(s) \geq 1 / 2 + \delta$.  Hence
\[
H_\kappa(s) \coloneqq \prod_p H_{\kappa, p}(s)
\]
converges normally there and defines a holomorphic function.  The Euler products give \eqref{eq:average-order-factorization} in the half-plane of absolute convergence.

At $s = 1$, telescoping yields
\[
\mathcal B_{\kappa, p}(1)
= 1 + \sum_{a \ge 1}\rksigma{\kappa}(p^a)
= \lim_{a \to \infty}\sigma_{-1}(p^a)^\kappa
= (1 - p^{-1})^{-\kappa}.
\]
Thus $H_{\kappa, p}(1) = 1$ for each prime $p$, and \eqref{eq:h-at-one} follows.
\end{proof}

We use the following consequence of Tenenbaum's Selberg--Delange theorem
\cite[Theorem~II.5.2]{Tenenbaum}.

\begin{proposition}\label{prop:tenenbaum-selberg-delange}
Let $z > 0$, and let
\[
F(s) = \sum_{n \geq 1}\frac{a_n}{n^s}
\]
have nonnegative coefficients.  Suppose that
\[
F(s) = \zeta(s)^zG(s),
\]
where $G$ is holomorphic and bounded on a half-plane
$\operatorname{Re}(s) \geq 1 - c$ for some $c > 0$.  If
\begin{equation}\label{eq:tenenbaum-coefficients}
\frac{s^zF(1 + s)}{1 + s} = \sum_{j \geq 0}\mu_j(z)s^j,
\qquad
\lambda_j(z) = \frac{\mu_j(z)}{\Gamma(z - j)},
\end{equation}
then, for each fixed integer $N \geq 0$,
\begin{equation}\label{eq:tenenbaum-selberg-delange}
\sum_{n \leq x}a_n
= x(\log x)^{z - 1}
\left\{
\sum_{0 \leq j \leq N}\frac{\lambda_j(z)}{(\log x)^j}
+ O_{z, N, G}\!\left((\log x)^{-N - 1}\right)
\right\}
\end{equation}
as $x \to \infty$.
\end{proposition}

The factorization in \Cref{lem:average-order-euler-product} first gives \eqref{eq:weighted-reciprocal-full-expansion}.  The factor
$(1 + z)^{-1}$ in \eqref{eq:psi-kappa-definition} comes from the Perron kernel in the Selberg--Delange expansion.
Abel summation then transfers \eqref{eq:weighted-reciprocal-full-expansion} to \eqref{eq:reciprocal-full-expansion}, while
the relation between \eqref{eq:phi-kappa-definition} and \eqref{eq:psi-kappa-definition} accounts for the shift in
the gamma factors.  A Mellin comparison at $s = 0$ determines the one
integration constant that Abel summation does not see.

\begin{theorem}[{\Cref{thm:intro-order-theory}}]\label{thm:reciprocal-average-order}
For a fixed real number $\kappa > 1$, let
\begin{equation}\label{eq:phi-kappa-definition}
\Phi_\kappa(z)
\coloneqq
\bigl(z\zeta(1 + z)\bigr)^\kappa H_\kappa(1 + z)
= \sum_{j \ge 0}c_{\kappa, j}z^j
\end{equation}
for $z$ in a neighborhood of $0$, where the branch is chosen so that
$\bigl(z\zeta(1 + z)\bigr)^\kappa = 1$ at $z = 0$, and let
\begin{equation}\label{eq:psi-kappa-definition}
\Psi_\kappa(z)
\coloneqq
\frac{\Phi_\kappa(z)}{1 + z}
= \sum_{j \ge 0}d_{\kappa, j}z^j.
\end{equation}
Then, for each integer $J \ge 0$, as $x \to \infty$, we have
\begin{equation}\label{eq:weighted-reciprocal-full-expansion}
\sum_{n \le x} n\,\rksigma{\kappa}(n)
=
x\sum_{j = 0}^{J}
\frac{d_{\kappa, j}}{\Gamma(\kappa - j)}
(\log x)^{\kappa - 1 - j}
+ O_{\kappa, J}\!\left(x(\log x)^{\kappa - 2 - J}\right),
\end{equation}
and
\begin{equation}\label{eq:reciprocal-full-expansion}
\sum_{n \le x}\rksigma{\kappa}(n)
=
\sum_{j = 0}^{J}
\frac{c_{\kappa, j}}{\Gamma(\kappa + 1 - j)}
(\log x)^{\kappa - j}
+ O_{\kappa, J}\!\left((\log x)^{\kappa - 1 - J}\right).
\end{equation}
Here $1 / \Gamma$ is understood as its entire continuation, so a displayed term vanishes whenever its denominator is a pole of $\Gamma$.

In particular, as $x \to \infty$, we have
\begin{align}
\sum_{n \le x} n\,\rksigma{\kappa}(n)
& =
\frac{x(\log x)^{\kappa - 1}}{\Gamma(\kappa)}
+ O_\kappa\!\left(x(\log x)^{\kappa - 2}\right)
\\
\sum_{n \le x}\rksigma{\kappa}(n)
& =
\frac{(\log x)^\kappa}{\Gamma(\kappa + 1)}
+ O_\kappa\!\left((\log x)^{\kappa - 1}\right)
.
\end{align}
Moreover, we have
\begin{equation}\label{eq:first-average-coefficients}
c_{\kappa, 0} = d_{\kappa, 0} = 1,
\qquad
c_{\kappa, 1} = \kappa\gamma + H_\kappa'(1),
\qquad
 d_{\kappa, 1} = \kappa\gamma + H_\kappa'(1) - 1.
\end{equation}
Consequently, as $x \to \infty$, the first two terms of the unweighted expansion are
\begin{equation}\label{eq:reciprocal-first-two-terms}
\sum_{n \le x}\rksigma{\kappa}(n)
=
\frac{(\log x)^\kappa}{\Gamma(\kappa + 1)}
+
\frac{\kappa\gamma + H_\kappa'(1)}{\Gamma(\kappa)}
(\log x)^{\kappa - 1}
+ O_\kappa\!\left(1 + (\log x)^{\kappa - 2}\right).
\end{equation}
\end{theorem}

\begin{proof}
We apply \Cref{prop:tenenbaum-selberg-delange} to the Dirichlet series
$\mathcal B_\kappa$.  By \Cref{lem:average-order-euler-product},
\[
\mathcal B_\kappa(s) = \zeta(s)^\kappa H_\kappa(s),
\]
where $H_\kappa$ is holomorphic and bounded on
$\operatorname{Re}(s) \geq 3 / 4$.  The coefficients $b_\kappa(n)$ are
nonnegative, so \Cref{prop:tenenbaum-selberg-delange} applies with $z = \kappa$.  Taking $N = J$
gives an error
$O_{\kappa, J}\!\left(x(\log x)^{\kappa - 2 - J}\right)$ as $x \to \infty$.  Moreover, we have
\[
\frac{z^\kappa\mathcal B_\kappa(1 + z)}{1 + z}
= \Psi_\kappa(z),
\]
so \eqref{eq:tenenbaum-coefficients} gives
$\mu_j(\kappa) = d_{\kappa, j}$ and hence the coefficient
$d_{\kappa, j} / \Gamma(\kappa - j)$.  This proves
\eqref{eq:weighted-reciprocal-full-expansion}.

We now transfer \eqref{eq:weighted-reciprocal-full-expansion} by Abel summation.  Put
\[
B_\kappa(x) \coloneqq \sum_{n \leq x}b_\kappa(n),
\qquad
A_\kappa(x) \coloneqq \sum_{n \leq x}\rksigma{\kappa}(n).
\]
Abel summation gives
\begin{equation}\label{eq:abel-average-order-transfer}
A_\kappa(x)
= \frac{B_\kappa(x)}x + \int_1^x\frac{B_\kappa(t)}{t^2}\,dt.
\end{equation}
Fix an integer $K > \kappa + 1$ and use \eqref{eq:weighted-reciprocal-full-expansion} through order
$K$.  Since
\[
\Phi_\kappa(z) = (1 + z)\Psi_\kappa(z),
\qquad
c_{\kappa, j} = d_{\kappa, j} + d_{\kappa, j - 1}
\qquad\text{where }d_{\kappa, -1} \coloneqq 0,
\]
splitting the integral at a fixed large point and integrating term by term in
\eqref{eq:abel-average-order-transfer} yields
\begin{equation}\label{eq:unweighted-expansion-with-constant}
A_\kappa(x)
=
\sum_{\substack{0 \leq j \leq K\\ j \neq \kappa}}
\frac{c_{\kappa, j}}{\Gamma(\kappa + 1 - j)}
(\log x)^{\kappa - j}
+ C_\kappa
+ O_{\kappa, K}\!\left((\log x)^{\kappa - 1 - K}\right)
\end{equation}
as $x \to \infty$.  Here the exclusion matters only when $\kappa$ is an integer; in that case
Abel summation leaves the constant term $C_\kappa$ undetermined.  The term
$C_\kappa$ also appears when $\kappa$ is not an integer, because the
integrated error is convergent at infinity.

It remains to determine $C_\kappa$.  For real $s > 0$, another application
of Abel summation gives the exact Mellin identity
\[
\mathcal R_\kappa(s)
= s\int_1^\infty A_\kappa(x)x^{-s - 1}\,dx
= s^{-\kappa}\Phi_\kappa(s).
\]
Let $m = \lfloor\kappa\rfloor$.  If $\kappa \notin \mathbb N$, insert
\eqref{eq:unweighted-expansion-with-constant}, subtract the terms
$j = 0, \ldots, m$, and let $s \to 0^+$.  The remaining powers and the error
have Mellin contribution tending to zero, whereas the constant contributes
$C_\kappa$.  Since
\[
s^{-\kappa}\Phi_\kappa(s)
- \sum_{j = 0}^{m}c_{\kappa, j}s^{j - \kappa}
\longrightarrow 0,
\]
we obtain $C_\kappa = 0$.  If $\kappa = m \in \mathbb N$, applying the Mellin
comparison after subtracting $j = 0, \ldots, m - 1$ gives
$C_\kappa = c_{\kappa, m}$.  Thus
\eqref{eq:unweighted-expansion-with-constant} contains exactly the terms
asserted in \eqref{eq:reciprocal-full-expansion}.  Taking $K$ larger than
an arbitrary prescribed $J$ and truncating proves \eqref{eq:reciprocal-full-expansion} for each
$J \geq 0$.

As $z \to 0$, since $H_\kappa(1) = 1$ and
\[
z\zeta(1 + z) = 1 + \gamma z + O(z^2),
\]
we have
\[
\Phi_\kappa(z)
= 1 + \bigl(\kappa\gamma + H_\kappa'(1)\bigr)z + O_\kappa(z^2).
\]
Dividing by $1 + z$ gives $d_{\kappa, 0}$ and $d_{\kappa, 1}$ in
\eqref{eq:first-average-coefficients}.  Taking $J = 0$ gives the one-term
estimates, while $J = 1$ gives \eqref{eq:reciprocal-first-two-terms}.  The
harmless $1$ in the final error term covers the range $1 < \kappa < 2$.
\end{proof}

The coefficients in \eqref{eq:reciprocal-full-expansion} can be recovered
from convergent local data.

\begin{remark}[Effective coefficients]
The coefficients in \Cref{thm:reciprocal-average-order} are effective because the Euler product for $H_\kappa$ may be differentiated term by term.  In particular, we have
\begin{equation}
H_\kappa'(1)
=
\sum_p (\log p)
\left\{
\frac{\kappa}{p - 1}
-
(1 - p^{-1})^\kappa
\sum_{a \ge 1}a\,\rksigma{\kappa}(p^a)
\right\}.
\end{equation}
As $p \to \infty$ through the primes, the summand is $O_\kappa((\log p) / p^2)$, so the prime sum converges absolutely.  Higher derivatives are obtained from analogous absolutely convergent prime sums.  Thus every $c_{\kappa, j}$ and $d_{\kappa, j}$ can be computed to certified precision by truncating these sums and bounding their tails.
\end{remark}

The Dirichlet series $\mathcal R_\kappa(s)$ also yields the convergent constant
that appears in the average order of $\ksigma{\kappa}$.  Following
\cite[Theorem~1.3]{Fan-Mits-Molnar}, we define
\begin{equation}\label{eq:mean-value-constant}
c(\kappa)
\coloneqq
\sum_{n \geq 1}\frac{\rksigma{\kappa}(n)}{n}.
\end{equation}
The local bounds show directly that the series in
\eqref{eq:mean-value-constant} converges, and they give a quantitative
truncation estimate.

\begin{theorem}\label{thm:mean-value}
For every real $\kappa > 1$, the Dirichlet series
\begin{equation}
\mathcal{R}_\kappa(s) \; \coloneqq \; \sum_{n \geq 1} \frac{\rksigma{\kappa}(n)}{n^s}
\end{equation}
converges absolutely for $\operatorname{Re}(s) > 0$.  For $\operatorname{Re}(s) > 1$, it satisfies
\[
\mathcal{R}_\kappa(s) = \frac{1}{\zeta(s)}\sum_{n \geq 1}\frac{\sigma_{-1}(n)^\kappa}{n^s}.
\]
In particular, $c(\kappa) = \mathcal{R}_\kappa(1)$ is finite and positive.  For every fixed $\delta \in (0, 1)$, we have
\begin{equation}\label{eq:partial-sum-with-1-n}
\sum_{n \leq x} \frac{\rksigma{\kappa}(n)}{n} \; = \; c(\kappa) \, + \, O_{\kappa, \delta}\!\left(x^{\delta - 1}\right)
\qquad\text{as }x \to \infty.
\end{equation}
\end{theorem}

\begin{proof}
The convolution identity $\rksigma{\kappa} = \sigma_{-1}^{\,\kappa}*\mu$ gives the factorization wherever both sides converge absolutely.  To establish convergence, \Cref{lem:discrete-mvt} gives
\[
0 < \rksigma{\kappa}(p^a)
\leq \frac{\kappa}{p^a}\left(\frac{p}{p - 1}\right)^{\!\kappa - 1}
\leq C_\kappa p^{-a}
\]
for each prime $p$ and $a \geq 1$.  Hence, as $p \to \infty$ through the primes and uniformly on each half-plane $\operatorname{Re}(s) \geq \delta > 0$, we have
\[
\sum_{a \geq 1}\frac{\rksigma{\kappa}(p^a)}{p^{as}}
= O_{\kappa, \delta}\!\left(p^{-1 - \delta}\right).
\]
The Euler product for $\mathcal R_\kappa(s)$ therefore converges absolutely for $\operatorname{Re}(s) > 0$, and the same is true of its Dirichlet series because its coefficients are non-negative.  Evaluating at $s = 1$ gives $c(\kappa) = \mathcal R_\kappa(1) > 0$.

We fix $\delta \in (0, 1)$.  For $n > x$ we have $n^{-1} \leq x^{\delta - 1}n^{-\delta}$, and therefore
\[
0 \leq c(\kappa) - \sum_{n \leq x}\frac{\rksigma{\kappa}(n)}n
= \sum_{n > x}\frac{\rksigma{\kappa}(n)}n
\leq x^{\delta - 1}\mathcal R_\kappa(\delta).
\]
This proves \eqref{eq:partial-sum-with-1-n}.
\end{proof}

At $\kappa = 2$, the reciprocal mean has an explicit form already obtained in our earlier paper.

\begin{corollary}[{\cite[Lemma~3.5, equations~(16)--(18)]{Fan-Mits-Molnar}}]\label{cor:mean-value-kappa-2}
For $\kappa = 2$, we have
\[
\mathcal{R}_2(s)
=
\frac{\zeta(s + 1)^2\zeta(s + 2)}{\zeta(2s + 2)}.
\]
In particular, we have
\[
c(2) \, = \, \frac{\zeta(2)^2 \zeta(3)}{\zeta(4)} \, = \, \frac{5}{2}\, \zeta(3),
\]
and
\begin{equation}\label{eq:partial-sum-kappa-2}
\sum_{n \leq x} \frac{\rksigma{2}(n)}{n} \, = \, \frac{5}{2}\zeta(3) \, + \, O\!\left(\frac{\log x}{x}\right).
\end{equation}
\end{corollary}


\subsection{Maximal order}

The normalized form of \eqref{eq:global-upper-bound} gives the upper half of the
maximal-order statement.  To obtain the matching lower half, we compare
the local factors with their limiting Mertens factors and choose a
sequence on which the prime-power exponents grow together.  We write
$\theta(x) \coloneqq \sum_{p \leq x}\log p$, we let $\pi(x)$ count the primes at
most $x$, and we put $\primorial{x} \coloneqq \prod_{p \leq x}p$.

\begin{lemma}\label{lem:local-mertens-factor}
For each prime $p$ and integer $a \geq 0$, we have
\begin{equation}\label{eq:local-mertens-factor}
\sigma_{-1}(p^a) \; = \; \frac{1 - p^{-(a + 1)}}{1 - p^{-1}} \; = \; \frac{p}{p - 1}\bigl(1 - p^{-(a + 1)}\bigr).
\end{equation}
In particular, $\sigma_{-1}(p^a) < p / (p - 1)$ for each $a$, and $\sigma_{-1}(p^a) \to p / (p - 1)$ as $a \to \infty$.
\end{lemma}

\begin{proof}
The first equality is the geometric series sum.  The second factors out $p / (p - 1)$, and the limit follows since $p^{-(a + 1)} \to 0$.
\end{proof}

By multiplicativity, \eqref{eq:local-mertens-factor} gives
\[
\sigma_{-1}(n) < M(n)
\coloneqq
\prod_{p \mid n}\frac p{p - 1}
\qquad\text{for }n > 1.
\]
We call $M(n)$ the \defi{Mertens product} of $n$.  Equality is approached only as the exponents of its prime divisors tend to infinity.

\begin{theorem}[{\Cref{thm:intro-order-theory}}]\label{thm:gronwall}
For each real $\kappa > 1$, we have
\begin{equation}\label{eq:gronwall}
\limsup_{n \to \infty}\rkRobinRatio{\kappa}{n} = 1.
\end{equation}
Equivalently, we have
\[
\limsup_{n \to \infty}
\frac{n\,\rksigma{\kappa}(n)}
{\kappa^{\omega(n)}\RobinScalePower{n}{\kappa - 1}}
= 1.
\]
\end{theorem}

The normalized form of \eqref{eq:global-upper-bound} supplies the upper bound.  Primorial powers with
slowly growing exponents asymptotically attain it.

\begin{proof}[Proof of \Cref{thm:gronwall}]
By \eqref{eq:global-upper-bound} and \eqref{eq:normalized-function}, $\rkrho{\kappa}(n) \leq \sigma_{-1}(n)$.  Classical Gr\"{o}nwall therefore gives
\[
\limsup_{n \to \infty}\rkRobinRatio{\kappa}{n} \leq 1.
\]
For the matching lower bound, we let $p_k$ denote the $k$th prime.  We choose a function $T\colon\bbZ_{\geq 1} \to \bbZ_{\geq 2}$ such that
\[
T(k) \to \infty,
\qquad
\log T(k) = o(\log p_k),
\]
and set
\[
n_k \; \coloneqq \; \primorial{p_k}^{\,T(k)}.
\]
By \eqref{eq:global-lower-bound} and \eqref{eq:normalized-function}, we have
\begin{equation}
\rkrho{\kappa}(n_k) \; \geq \; \prod_{p \leq p_k} \sigma_{-1}(p^{T(k) - 1}) \; = \; \MertensProduct{p_k} \cdot \prod_{p \leq p_k}\bigl(1 - p^{-T(k)}\bigr),
\end{equation}
where the equality uses \Cref{lem:local-mertens-factor}.  Mertens' third theorem gives
$\MertensProduct{p_k} \sim e^\gamma\log p_k$.  The factor
$\prod_{p \leq p_k}(1 - p^{-T(k)})$ tends to $1$, because
\[
0 \; \leq \; - \log \prod_{p \leq p_k}\bigl(1 - p^{-T(k)}\bigr) \; \leq \; 2 \sum_{p} p^{-T(k)} \; \leq \; C\, 2^{-T(k)} \; \longrightarrow \; 0.
\]
Here we used $-\log(1 - x) \leq 2x$ for $x \leq 1 / 2$.  Hence
\[
\rkrho{\kappa}(n_k) \geq (1 + o(1))e^\gamma\log p_k.
\]
Moreover, we have
\[
\log n_k = T(k)\theta(p_k) \sim T(k)p_k.
\]
The condition $\log T(k) = o(\log p_k)$ therefore gives
$\log\log n_k = \log p_k + o(\log p_k)$.  After division by
$\RobinScale{n_k}$, the lower limit along this sequence is at least
$1$.  Together with \eqref{eq:global-upper-bound} and \eqref{eq:normalized-function}, the lower-limit estimate establishes
\eqref{eq:gronwall}.
\end{proof}

\subsection{A uniform explicit upper bound}

\Cref{thm:gronwall} is asymptotic.  For a uniform pointwise bound,
we compare the Mertens product of an integer with an initial prime
product and then relate the relevant primorial to $\log n$.  \Cref{prop:dusart,prop:schoenfeld} provide the required explicit input.

\begin{proposition}[{\cite[Theorem~5.9, p.~245]{Dusart}}]\label{prop:dusart-source}
For every $x \geq 2{,}278{,}382$, we have
\begin{equation}\label{eq:dusart-direct-product}
\frac{e^{-\gamma}}{\log x}
\left(1 - \frac{0.2}{(\log x)^3}\right)
< \prod_{p \leq x}(1 - p^{-1})
< \frac{e^{-\gamma}}{\log x}
\left(1 + \frac{0.2}{(\log x)^3}\right),
\end{equation}
and
\begin{equation}\label{eq:dusart-reciprocal-product}
\MertensProduct{x}
< e^\gamma\log x\left(1 + \frac{0.2}{(\log x)^3}\right).
\end{equation}
\end{proposition}

\begin{proposition}\label{prop:dusart}
For all $x \geq \MertensThreshold$, we have
\begin{equation}\label{eq:mertens-explicit-extension}
\MertensProduct{x} \; < \; e^\gamma \log x \parent{1 + \frac{\MertensCorrection}{(\log x)^2}}.
\end{equation}
\end{proposition}

\begin{proof}
For $\MertensThreshold \leq x < 2{,}278{,}382$, directed interval arithmetic verifies \eqref{eq:mertens-explicit-extension} at all $165{,}866$ prime endpoints.  The product is constant between consecutive primes, whereas the right-hand side is increasing, so these checks cover the entire interval.  The largest required correction coefficient is less than $0.109511$ and occurs at $x = 24{,}137$.

For $x \geq 2{,}278{,}382$, the lower bound in
\eqref{eq:dusart-direct-product} gives
\[
\prod_{p \leq x}(1 - p^{-1})
>
\frac{e^{-\gamma}}{\log x}
\parent{1 - \frac{0.2}{(\log x)^3}}.
\]
Put $\delta = 0.2 / (\log x)^3$.  Since $\log x > 4$, we have $0 < \delta < 1 / 2$, and therefore
\[
\parent{1 - \delta}^{-1} < 1 + 2\delta
< 1 + \frac{\MertensCorrection}{(\log x)^2}.
\]
Taking reciprocals proves \eqref{eq:mertens-explicit-extension}.
\end{proof}

We now require a result due to Schoenfeld.

\begin{proposition}[{\cite[Corollary~2*]{Schoenfeld}}]\label{prop:schoenfeld}
For all $x \geq \MertensThreshold$, we have
\begin{equation}\label{eq:schoenfeld-lower}
\theta(x) \; > \; x\parent{1 - \frac{1}{8\log x}}.
\end{equation}
\end{proposition}

The largest prime factor $P$ of $n$ need not satisfy $\theta(P) \leq \log n$.
For a nonnegative real number $y$, we define the \defi{primorial residual}
$\calP^*(y)$ to be the largest prime $\ell$ for which
$\primorial{\ell} \leq y$, if such a prime exists, and we put
$\calP^*(y) = 0$ otherwise.  We put $x = \calP^*(n)$.  Then
$\theta(x) \leq \log n$ and $\omega(n) \leq \pi(x)$, which together compare
the prime support of $n$ with the initial Mertens product through $x$.

\begin{proposition}\label{prop:mertens-transfer}
For each integer $n \geq e^{\PrimorialLogThreshold}$, we have
\begin{equation}\label{eq:mertens-transfer}
M(n) \; < \; \RobinScale{n} + \frac{\TransferCorrection}{\log\log n}.
\end{equation}
\end{proposition}

\begin{proof}
Put $x = \calP^*(n)$.  Then $x \geq \MertensThreshold$, because
$\theta(\MertensThreshold) < \PrimorialLogThreshold$, and
$\theta(x) \leq \log n$.  Moreover $\omega(n) \leq \pi(x)$: otherwise the
product of the distinct prime divisors of $n$ would be at least the product of
the first $\pi(x) + 1$ primes, which exceeds $n$.  Since $p / (p - 1)$ decreases
with $p$, it follows that
\[
M(n) \leq \MertensProduct{x}.
\]
By \eqref{eq:mertens-explicit-extension} and \eqref{eq:schoenfeld-lower}, we have
\begin{equation}\label{eq:mertens-transfer-prebounds}
M(n) < e^\gamma L\left(1 + \frac{\MertensCorrection}{L^2}\right),
\qquad
\log\log n \geq L + \log\left(1 - \frac1{8L}\right)
> L - \frac1{7.92L},
\end{equation}
where $L \coloneqq \log x$.  The lower bound for $\log\log n$ in
\eqref{eq:mertens-transfer-prebounds} follows by putting
\[
h(L) \coloneqq \log\left(1 - \frac1{8L}\right) + \frac1{7.92L}.
\]
Its derivative has the sign of $1 - 0.08L$.  Thus $h$ increases up to $L = 12.5$ and then decreases to $0$.  At the left endpoint, we have
\[
h(\log\MertensThreshold) > 4.70 \times 10^{-5},
\]
and hence $h(L) > 0$ for every $L \geq \log\MertensThreshold$.

We put $y \coloneqq L - 1 / (7.92L)$.  Since $\log\log n \geq y$ and the function $z \mapsto e^\gamma z + \TransferCorrection / z$ is increasing for $z \geq \log \MertensThreshold - 1 / (7.92\log\MertensThreshold)$, it is enough to prove
\begin{equation}\label{eq:mertens-transfer-target}
e^\gamma L\left(1 + \frac{\MertensCorrection}{L^2}\right)
< e^\gamma y + \frac{\TransferCorrection}{y}.
\end{equation}
A stronger sufficient condition for \eqref{eq:mertens-transfer-target} is
obtained by replacing $\TransferCorrection / y$ with the smaller quantity
$\TransferCorrection / L$.  After cancellation, the sufficient condition for
\eqref{eq:mertens-transfer-target} becomes
\begin{equation}\label{eq:mertens-transfer-numerical}
e^\gamma\left(\MertensCorrection + \frac1{7.92}\right) < \TransferCorrection.
\end{equation}
The left-hand side of \eqref{eq:mertens-transfer-numerical} is
$0.4199459\ldots$.
\end{proof}

The threshold in \Cref{thm:unconditional-bound} uses the finite Mertens extension in
\Cref{prop:dusart}, together with the verified inequality
$\theta(\MertensThreshold) < \PrimorialLogThreshold$ for the primorial threshold.

\begin{theorem}\label{thm:unconditional-bound}
For every real $\kappa > 1$ and each $n \ge e^{\PrimorialLogThreshold}$, we have
\begin{equation}
\rksigma{\kappa}(n)
< \frac{\kappa^{\omega(n)}}n
\left(\RobinScale{n} + \frac{\TransferCorrection}{\log\log n}\right)^{\kappa - 1}.
\end{equation}
\end{theorem}

Combining \eqref{eq:global-upper-bound}, \eqref{eq:local-mertens-factor}, and \eqref{eq:mertens-transfer} gives \Cref{thm:unconditional-bound}.

\begin{proof}[Proof of \Cref{thm:unconditional-bound}]
By \eqref{eq:global-upper-bound} and \eqref{eq:local-mertens-factor}, we have
\begin{equation}
\rksigma{\kappa}(n) \; \leq \; \frac{\kappa^{\omega(n)}}{n}\, \sigma_{-1}(n)^{\kappa - 1} \; \leq \; \frac{\kappa^{\omega(n)}}{n}\, M(n)^{\kappa - 1}.
\end{equation}
We now apply \eqref{eq:mertens-transfer} and raise it to the positive power $\kappa - 1$.
\end{proof}

\section{Parameter monotonicity and the reciprocal Robin criterion}\label{sec:parameter-monotonicity}

\Cref{sec:order-theory} places each fixed parameter on the classical Robin
scale.  At a prime power, the normalized local factor is a power mean.
The representation in \Cref{lem:power-mean} orders the family, transfers Robin's inequality, propagates
violations upward, and recovers the classical criterion along any unbounded
parameter set.

\begin{lemma}[{\Cref{thm:intro-robin-criterion}}]\label{lem:power-mean}
For each integer $n > 1$, the function $\kappa \mapsto \rkrho{\kappa}(n)$ is
strictly increasing on $(1, \infty)$, with limit
\[
\lim_{\kappa \to \infty}\rkrho{\kappa}(n) = \sigma_{-1}(n).
\]
In particular, $\rkrho{\kappa}(n) < \sigma_{-1}(n)$ for every finite
$\kappa > 1$.
\end{lemma}

\begin{proof}
By multiplicativity, it suffices to treat a prime power $p^a$, and we write
\[
S \coloneqq \sigma_{-1}(p^a),
\qquad
u \coloneqq \frac{\sigma_{-1}(p^{a - 1})}{\sigma_{-1}(p^a)} \in (0, 1).
\]
Since $S - \sigma_{-1}(p^{a - 1}) = p^{-a}$, \eqref{eq:rksigma-at-prime-powers} and \Cref{def:normalized-reciprocal-lcm-power} give
\begin{equation}\label{eq:local-power-mean-representation}
\begin{aligned}
\frac{\rkrho{\kappa}(p^a)}{S}
& = \left(\frac{1 - u^\kappa}{\kappa(1 - u)}\right)^{1 / (\kappa - 1)} \\
& = \left(\frac1{1 - u}\int_u^1 t^{\kappa - 1}\,dt\right)^{1 / (\kappa - 1)}.
\end{aligned}
\end{equation}
The right-hand side of \eqref{eq:local-power-mean-representation} is the power mean of order $\kappa - 1$ of the nonconstant function
$t$ under normalized Lebesgue measure on $[u, 1]$.  Strict power-mean
monotonicity makes it strictly increasing in $\kappa$, and its limit is
the essential supremum of $t$, namely $1$.  Thus
\[
\rkrho{\kappa}(p^a) \nearrow \sigma_{-1}(p^a).
\]
Multiplication over $p^a \parallel n$ proves \Cref{lem:power-mean}.
\end{proof}

The strict local comparison first transfers Robin's inequality to each
finite parameter.

\begin{theorem}\label{thm:conditional-robin}
Assume the Riemann hypothesis.  Then, for every real $\kappa > 1$ and every
$n > 5040$, we have
\begin{equation}
\rksigma{\kappa}(n)
< \frac{\kappa^{\omega(n)}}n\RobinScalePower{n}{\kappa - 1},
\end{equation}
or equivalently $\rkrho{\kappa}(n) < \RobinScale{n}$.
\end{theorem}

\begin{proof}
By \Cref{lem:power-mean},
$\rkrho{\kappa}(n) < \sigma_{-1}(n)$ for finite $\kappa > 1$.  Under the
Riemann hypothesis, \eqref{eq:classical-robin} gives
$\sigma_{-1}(n) < \RobinScale{n}$ for $n > 5040$.  Combining \Cref{lem:power-mean} with \eqref{eq:classical-robin} proves
\Cref{thm:conditional-robin}.
\end{proof}

Conversely, a violation cannot disappear farther up the family.

\begin{corollary}\label{cor:counterexamples-propagate}
If $\rkrho{\kappa_0}(n_0) \geq \RobinScale{n_0}$ for some
$\kappa_0 > 1$ and $n_0 > 1$, then
$\rkrho{\kappa}(n_0) \geq \RobinScale{n_0}$ for every $\kappa \geq \kappa_0$.
\end{corollary}

\begin{proof}
The conclusion follows from the monotonicity in \Cref{lem:power-mean}.
\end{proof}


\subsection{The family-wide criterion}

No finite parameter reaches $\sigma_{-1}$, so \Cref{thm:conditional-robin} does
not by itself recover the Riemann hypothesis.  Equivalence returns when
$\kappa$ can grow without bound, and every unbounded parameter set is
sufficient.

\begin{theorem}[{\Cref{thm:intro-robin-criterion}}]\label{thm:reciprocal-robin-criterion}
For an unbounded subset $\mathcal K \subset (1, \infty)$, the following assertions are equivalent.
\begin{enumerate}
\renewcommand{\theenumi}{(\roman{enumi})}
\renewcommand{\labelenumi}{\textup{\theenumi}}
\item\label{item:criterion-rh} The Riemann hypothesis holds.
\item\label{item:criterion-all-kappa} For each integer $n > 5040$ and every real $\kappa > 1$, we have
\[
\rkrho{\kappa}(n) < \RobinScale{n}.
\]
\item\label{item:criterion-unbounded-set} For each integer $n > 5040$ and every $\kappa \in \mathcal K$, we have
\[
\rkrho{\kappa}(n) < \RobinScale{n}.
\]
\item\label{item:criterion-eventual} There is an integer $n_0$ such that, for each integer $n > n_0$ and every $\kappa \in \mathcal K$, we have
\[
\rkrho{\kappa}(n) < \RobinScale{n}.
\]
\end{enumerate}
In particular, one may take $\mathcal K = \{2, 3, 4, \ldots\}$.
\end{theorem}

\begin{proof}
The implication \ref{item:criterion-rh}$\Rightarrow$\ref{item:criterion-all-kappa} is
\Cref{thm:conditional-robin}, while
\ref{item:criterion-all-kappa}$\Rightarrow$\ref{item:criterion-unbounded-set}$\Rightarrow$\ref{item:criterion-eventual} is
immediate.

We suppose that the Riemann hypothesis is false.  Robin proved that
\[
\sigma_{-1}(n) > \RobinScale{n}
\]
for infinitely many integers $n$ \cite[\S4]{Robin}.  We fix one such
integer.  By \Cref{lem:power-mean}, we have
\[
\rkrho{\kappa}(n) \nearrow \sigma_{-1}(n)
\qquad\text{as }\kappa \to \infty,
\]
so $\rkrho{\kappa}(n) > \RobinScale{n}$ for every $\kappa$ beyond some
threshold $K_n$.  Since $\mathcal K$ is unbounded, it contains such a
parameter.  The parameter may depend on $n$.  Nevertheless, the infinitely many
classical violations contradict \ref{item:criterion-eventual} for each $n_0$.
Hence \ref{item:criterion-eventual} implies \ref{item:criterion-rh}.
\end{proof}

\section{Unconditional results at bounded parameters}\label{sec:unconditional-results}

\Cref{thm:reciprocal-robin-criterion} requires an unbounded parameter set.
\Cref{thm:bounded-eventual} shows why: every bounded parameter range is
eventually good unconditionally.  A special feature at $\kappa = 5$ gives a
shorter direct argument for that endpoint; we record it in
\Cref{rem:kappa-five-square-support}.  The unnormalized reciprocal extremizers
do not control the Robin quotient because of a residual factor depending on
$\omega(n)$; \Cref{app:unnormalized-extremal} explains that obstruction separately.

\subsection{Normalized extremizers and endpoint transfer}

We now turn from parameter monotonicity to the extremal problem governing a fixed bounded parameter range.
Fix a real $\kappa_0 > 1$ and put
\[
Q_{\kappa_0}(n) \coloneqq \rkrho{\kappa_0}(n).
\]
We optimize the normalized function that occurs in \eqref{eq:intro-reciprocal-robin}.

\begin{definition}\label{def:normalized-colossally-abundant}
For $\epsilon > 0$, an integer $N$ is \defi{normalized
$\kappa_0$-colossally abundant for parameter $\epsilon$} if
\begin{equation}\label{eq:normalized-ca-definition}
\frac{Q_{\kappa_0}(n)}{n^\epsilon} \leq \frac{Q_{\kappa_0}(N)}{N^\epsilon}
\qquad\text{for }n \geq 1.
\end{equation}
\end{definition}

For a prime $p$ and an integer $a \geq 0$, we write
\[
\mathcal L_{p, \epsilon}(a)
\coloneqq \log Q_{\kappa_0}(p^a) - a\epsilon\log p.
\]
Because $Q_{\kappa_0}$ is positive and multiplicative, the optimization separates
prime by prime.  \Cref{lem:canonical-endpoints} fixes the larger choice at a tie and orders
the resulting extremizers.

\begin{lemma}\label{lem:canonical-endpoints}
For every $\epsilon > 0$ and every prime $p$, let $a_p(\epsilon)$ be the
largest nonnegative integer maximizing $\mathcal L_{p, \epsilon}(a)$, and put
\[
\mathcal N_{\kappa_0}(\epsilon) \coloneqq \prod_p p^{a_p(\epsilon)}.
\]
Then the product is finite, and $\mathcal N_{\kappa_0}(\epsilon)$ is the largest normalized
$\kappa_0$-colossally abundant integer for parameter $\epsilon$.  Moreover:
\begin{enumerate}
\item if $0 < \epsilon_2 < \epsilon_1$, then
$a_p(\epsilon_2) \geq a_p(\epsilon_1)$ for each prime $p$, and hence
$\mathcal N_{\kappa_0}(\epsilon_2) \geq \mathcal N_{\kappa_0}(\epsilon_1)$;
\item\label{item:canonical-endpoints-local-finiteness} on every compact subinterval of $(0, \infty)$, the function
$\mathcal N_{\kappa_0}(\epsilon)$ takes only finitely many values;
\item $\mathcal N_{\kappa_0}(\epsilon) \to \infty$ as $\epsilon \to 0^+$.
\end{enumerate}
If $N < N'$ are consecutive distinct values of $\mathcal N_{\kappa_0}(\epsilon)$, then some
parameter $\epsilon_0 > 0$ supports both: each maximizes
$Q_{\kappa_0}(n)n^{-\epsilon_0}$ over all $n \geq 1$.
\end{lemma}

\begin{proof}
For fixed $p$ and $\epsilon$, \Cref{lem:power-mean} gives
$Q_{\kappa_0}(p^a) < p / (p - 1)$, and therefore
\begin{equation}\label{eq:canonical-local-decay}
Q_{\kappa_0}(p^a)p^{-a\epsilon} < \frac p{p - 1}p^{-a\epsilon} \longrightarrow 0
\qquad\text{as }a \to \infty.
\end{equation}
Thus a largest local maximizer exists.  For all sufficiently large $p$,
\eqref{eq:canonical-local-decay} makes every positive exponent worse than exponent zero.  Hence
$a_p(\epsilon) = 0$ for all but finitely many primes, and multiplicativity makes
$\mathcal N_{\kappa_0}(\epsilon)$ the largest global maximizer.

We let $\epsilon_1 > \epsilon_2 > 0$ and abbreviate
$a_i = a_p(\epsilon_i)$.  The two maximizing inequalities are
\begin{equation}\label{eq:canonical-monotonicity-maxima}
\mathcal L_{p, \epsilon_1}(a_1) \geq \mathcal L_{p, \epsilon_1}(a_2),
\qquad
\mathcal L_{p, \epsilon_2}(a_2) \geq \mathcal L_{p, \epsilon_2}(a_1).
\end{equation}
Adding \eqref{eq:canonical-monotonicity-maxima} componentwise yields
\[
(a_2 - a_1)(\epsilon_1 - \epsilon_2)\log p \geq 0,
\]
so $a_2 \geq a_1$.

We fix $0 < \delta < M$.  If $\epsilon \in [\delta, M]$ and
$a_p(\epsilon) > 0$, then
\begin{equation}\label{eq:canonical-compact-prime-bound}
1 \leq Q_{\kappa_0}(p^{a_p(\epsilon)})p^{-a_p(\epsilon)\epsilon}
< \frac p{p - 1}p^{-\delta}.
\end{equation}
By \eqref{eq:canonical-compact-prime-bound}, only finitely many primes can occur.  For each of those primes,
the argument leading to \eqref{eq:canonical-compact-prime-bound}, with $a\delta$ in the exponent, bounds every locally
maximizing exponent.  Only finitely many affine functions
$\mathcal L_{p, \epsilon}(a)$ can therefore participate on $[\delta, M]$.
Their finitely many crossings prove item~\ref{item:canonical-endpoints-local-finiteness} of \Cref{lem:canonical-endpoints}.

Finally, with $r \coloneqq \kappa_0 - 1$ and $x \coloneqq p^{-1}$, we have
\[
Q_{\kappa_0}(p)^r
= \frac{(1 + x)^{\kappa_0} - 1}{\kappa_0 x} > 1
\]
by strict convexity of $t \mapsto t^{\kappa_0}$.  Hence
$Q_{\kappa_0}(p)p^{-\epsilon} > 1$ for all sufficiently small $\epsilon > 0$.
Every fixed prime therefore eventually divides $\mathcal N_{\kappa_0}(\epsilon)$, proving that
$\mathcal N_{\kappa_0}(\epsilon) \to \infty$.  At a transition between consecutive values, local
finiteness and continuity force the extremizers on the two sides to tie.
\end{proof}

\begin{definition}\label{def:canonical-normalized-endpoints}
For fixed $\kappa_0$, the distinct values of $\mathcal N_{\kappa_0}(\epsilon)$ are the
\defi{canonical normalized $\kappa_0$-endpoints}.  Two such endpoints are
\defi{consecutive} when no other distinct value lies strictly between them.
\end{definition}

The normalized Robin quotient between two consecutive endpoints is controlled
by convexity in $\log n$.

\begin{lemma}\label{lem:endpoint-principle}
For consecutive canonical normalized $\kappa_0$-endpoints $N < N'$, if
\begin{equation}\label{eq:endpoint-principle-hypotheses}
Q_{\kappa_0}(N) < \RobinScale{N}
\qquad\hbox{and}\qquad
Q_{\kappa_0}(N') < \RobinScale{N'},
\end{equation}
then $Q_{\kappa_0}(n) < \RobinScale{n}$ for each integer $N \leq n \leq N'$.
\end{lemma}

\begin{proof}
By \Cref{lem:canonical-endpoints}, there is a common supporting parameter
$\epsilon > 0$ for which both integers maximize
\eqref{eq:normalized-ca-definition}.  Hence, for $N \leq n \leq N'$, we have
\[
\log Q_{\kappa_0}(n) \leq C + \epsilon\log n
\]
with equality at both endpoints, where
$C \coloneqq \log Q_{\kappa_0}(N) - \epsilon\log N$.  Put $t \coloneqq \log n$.  Then
\[
\log\frac{Q_{\kappa_0}(n)}{\RobinScale{n}}
\leq C - \gamma + \epsilon t - \log\log t.
\]
The function $t \mapsto \epsilon t - \log\log t$ is strictly convex for $t > 1$,
since its second derivative is
\[
\frac{\log t + 1}{t^2(\log t)^2} > 0.
\]
Its maximum on $[\log N, \log N']$ therefore occurs at an endpoint, where
\eqref{eq:endpoint-principle-hypotheses} makes the right-hand side negative.
\end{proof}

\subsection{Effective repeated support and the eventual theorem}

\Cref{lem:endpoint-principle} reduces the eventual inequality to the canonical endpoints, so it remains to control those endpoints effectively.

\begin{lemma}\label{lem:effective-repeated-support}
Let $N$ be a canonical normalized $\kappa_0$-endpoint, supported by $\epsilon$, and
let $P$ be its largest prime factor.  Put
\[
A_{\kappa_0} \coloneqq \frac{3(\kappa_0 + 3)}4.
\]
Then every prime $p < P$ divides $N$, and
\[
v_p(N) \geq 2
\qquad\text{for }A_{\kappa_0} \leq p \leq \frac{P - 1}{3}.
\]
Consequently, we have
\begin{equation}\label{eq:repeated-support-log-lower}
\log N \geq
\theta(P) + \theta\!\left(\frac{P - 1}{3}\right) - \theta(A_{\kappa_0}),
\end{equation}
and
\begin{equation}\label{eq:repeated-support-mertens-upper}
Q_{\kappa_0}(N) < \MertensProduct{P}.
\end{equation}
\end{lemma}

\begin{proof}
For fixed $a \geq 1$, write $S_j(p) \coloneqq \sigma_{-1}(p^j)$.  \eqref{eq:local-power-mean-representation} gives
\[
Q_{\kappa_0}(p^a)^{\kappa_0 - 1}
= \int_0^1\bigl((1 - t)S_{a - 1}(p) + tS_a(p)\bigr)^{\kappa_0 - 1}\,dt.
\]
Both endpoints decrease strictly with $p$, so $Q_{\kappa_0}(p^a)$ does as well.  If
$P^a \parallel N$ and $p < P$, then
\[
Q_{\kappa_0}(p^a)p^{-a\epsilon}
> Q_{\kappa_0}(P^a)P^{-a\epsilon} \geq 1,
\]
so every prime below $P$ occurs in $N$.

We now put $r \coloneqq \kappa_0 - 1$ and $x \coloneqq 1 / p$.  For exponent one, we have
\[
Q_{\kappa_0}(p)^r = \frac1x\int_0^x(1 + s)^r\,ds.
\]
Since $(1 + s)^r \leq e^{rs}$ and the exponential is convex, we have
\[
\frac1x\int_0^xe^{rs}\,ds
\leq \frac{1 + e^{rx}}2
= e^{rx / 2}\cosh(rx / 2).
\]
Using $\log\cosh z \leq z^2 / 2$ gives
\begin{equation}\label{eq:qk-p-upper}
\log Q_{\kappa_0}(p) \leq \frac{x}{2} + \frac{r x^2}{8}.
\end{equation}
The factor $Q_{\kappa_0}(p^2)$ is the power mean of order $r$ on
$[1 + x, 1 + x + x^2]$, so $Q_{\kappa_0}(p^2) > 1 + x$.  Hence
\begin{equation}\label{eq:qk-p2-lower}
\log Q_{\kappa_0}(p^2) \geq \log(1 + x) \geq x - \frac{x^2}{2}.
\end{equation}
Subtracting \eqref{eq:qk-p-upper} from \eqref{eq:qk-p2-lower} yields
\[
\log\frac{Q_{\kappa_0}(p^2)}{Q_{\kappa_0}(p)}
\geq \frac{1}{2p} - \frac{\kappa_0 + 3}{8p^2}
\geq \frac1{3p}
\qquad\text{for }p \geq A_{\kappa_0}.
\]

If $p \parallel N$, local optimality of exponent one against exponent two
therefore gives
\begin{equation}\label{eq:epsilon-lower-first-power}
\epsilon \geq \frac1{3p\log p}
\qquad\text{for }p \geq A_{\kappa_0}.
\end{equation}
On the other hand, the positive exponent at $P$ beats exponent zero, and
\Cref{lem:power-mean} gives
\begin{equation}\label{eq:epsilon-upper-largest-prime}
\epsilon
< \frac{\log(P / (P - 1))}{\log P}
< \frac1{(P - 1)\log P}.
\end{equation}
For $p \leq (P - 1) / 3$, \eqref{eq:epsilon-lower-first-power} contradicts
\eqref{eq:epsilon-upper-largest-prime}.  Thus every prime in the asserted
range occurs at least twice.  Summing logarithms gives
\eqref{eq:repeated-support-log-lower}.  Finally,
\Cref{lem:power-mean} gives $Q_{\kappa_0}(p^a) < p / (p - 1)$ prime by prime, proving
\eqref{eq:repeated-support-mertens-upper}.
\end{proof}

\begin{proposition}\label{prop:bounded-parameter-tail}
Put
\begin{equation}\label{eq:bounded-tail-C}
C_{\kappa_0} \coloneqq \theta\!\left(\frac{3(\kappa_0 + 3)}4\right),
\end{equation}
and
\begin{equation}\label{eq:bounded-tail-X}
X_{\kappa_0} \coloneqq
\max\left\{
2{,}278{,}382,
\frac{100}{31}\left(C_{\kappa_0} + \frac13\right)
\right\}.
\end{equation}
Every canonical normalized $\kappa_0$-endpoint whose largest prime factor $P$
satisfies $P > X_{\kappa_0}$ obeys
\begin{equation}\label{eq:bounded-tail-robin}
Q_{\kappa_0}(N) < \RobinScale{N}.
\end{equation}
\end{proposition}

\begin{proof}
By \Cref{lem:effective-repeated-support}, it is enough to prove
\[
\MertensProduct{P}
< e^\gamma\log\left(
\theta(P) + \theta\!\left(\frac{P - 1}{3}\right) - C_{\kappa_0}
\right).
\]
We suppose $P \geq 2{,}278{,}382$.  Then both $P$ and $(P - 1) / 3$ exceed
$\MertensThreshold$.  Since $\log\MertensThreshold > 9$, \eqref{eq:schoenfeld-lower} gives
\[
\theta(P) + \theta\!\left(\frac{P - 1}{3}\right) - C_{\kappa_0}
> \frac{71}{54}P - \frac13 - C_{\kappa_0}.
\]
By \eqref{eq:dusart-reciprocal-product}, we have
\[
\MertensProduct{P}
< e^\gamma\log P\left(1 + \frac{0.2}{(\log P)^3}\right)
= e^\gamma\left(\log P + \frac{0.2}{(\log P)^2}\right).
\]
It consequently suffices that
\[
\frac{71}{54}P - \frac13 - C_{\kappa_0}
> P\exp\!\left(\frac{0.2}{(\log P)^2}\right).
\]
Again using $\log P > 14$, we have
\[
\exp\!\left(\frac{0.2}{(\log P)^2}\right)
< e^{1 / 980} < \frac{980}{979} < \frac{1003}{1000}.
\]
Moreover
\[
\frac{71}{54} - \frac{1003}{1000}
= \frac{8419}{27{,}000} > \frac{31}{100}.
\]
Thus \eqref{eq:bounded-tail-robin} follows from
\[
\frac{31}{100}P > C_{\kappa_0} + \frac13,
\]
which is exactly the second term in the maximum defining $X_{\kappa_0}$ in \eqref{eq:bounded-tail-X}.
\end{proof}

\begin{theorem}[{\Cref{thm:intro-unconditional}}]\label{thm:bounded-eventual}
For every real $\kappa_0 > 1$, there is an explicit finite integer $N_{\kappa_0}$ such
that
\begin{equation}\label{eq:bounded-eventual-robin}
\rkrho{\kappa}(n) < \RobinScale{n}
\end{equation}
for each integer $n \geq N_{\kappa_0}$ and every real $1 < \kappa \leq \kappa_0$.

More explicitly, let $q_{\kappa_0}$ be the smallest prime greater than $X_{\kappa_0}$ and put
\begin{equation}\label{eq:eta-kappa0-definition}
\eta_{\kappa_0}(q_{\kappa_0})
\coloneqq
\max_{a \geq 1}
\frac{\log Q_{\kappa_0}(q_{\kappa_0}^a)}{a\log q_{\kappa_0}}.
\end{equation}
Then the canonical endpoint
\begin{equation}\label{eq:E-kappa0-definition}
E_{\kappa_0} \coloneqq \mathcal N_{\kappa_0}\bigl(\eta_{\kappa_0}(q_{\kappa_0})\bigr)
\end{equation}
has largest prime factor $q_{\kappa_0}$, and \eqref{eq:bounded-eventual-robin} holds for each $n \geq E_{\kappa_0}$.
In particular, the completely explicit cutoff
\begin{equation}\label{eq:explicit-NK}
N_{\kappa_0}^{\rm eff}
\coloneqq
\left\lceil
\exp\left(
\frac{\log\displaystyle\MertensProduct{q_{\kappa_0}}\,\log q_{\kappa_0}}
{\log Q_{\kappa_0}(q_{\kappa_0})}
\right)
\right\rceil
\end{equation}
is admissible.
\end{theorem}

\begin{proof}
By \Cref{lem:power-mean}, it suffices to prove the assertion at $\kappa = \kappa_0$.
For fixed $a$, the quantity
\[
p \longmapsto \frac{\log Q_{\kappa_0}(p^a)}{a\log p}
\]
is strictly decreasing on the primes: its numerator is positive and strictly
decreasing, while its denominator is strictly increasing.  For each fixed
prime $p$, moreover, we have
\[
0 < \frac{\log Q_{\kappa_0}(p^a)}{a\log p}
< \frac{\log(p / (p - 1))}{a\log p} \longrightarrow 0,
\]
whereas the $a = 1$ term is positive.  Thus the maximum defining
$\eta_{\kappa_0}(p)$ is attained after a finite search.  It follows as well that
$\eta_{\kappa_0}(p)$ is strictly decreasing with $p$.  At
$\epsilon = \eta_{\kappa_0}(q_{\kappa_0})$, the largest-at-a-tie convention includes $q_{\kappa_0}$, and
no larger prime occurs.  Thus $E_{\kappa_0}$ has largest prime factor exactly $q_{\kappa_0}$ and
is good by \Cref{prop:bounded-parameter-tail}.  Every later endpoint retains
$q_{\kappa_0}$ by \Cref{lem:canonical-endpoints}, so every later endpoint is good as
well.  \Cref{lem:endpoint-principle} now gives \eqref{eq:bounded-eventual-robin} for all
$n \geq E_{\kappa_0}$.

For each fixed integer $\kappa_0$, the finite search in \eqref{eq:eta-kappa0-definition} is effective: the local
quantities have algebraic values whose relevant finite comparisons can be made
exactly.  This is the effective regime used in \Cref{thm:unconditional-interval}.  It remains only
to justify the explicit upper bound.  If
$p^{a_p} \parallel E_{\kappa_0}$, then local optimality against exponent zero gives
\[
a_p\eta_{\kappa_0}(q_{\kappa_0})\log p
< \log\frac p{p - 1}.
\]
Summing over $p \leq q_{\kappa_0}$ gives
\[
\log E_{\kappa_0}
< \frac{\log\displaystyle\MertensProduct{q_{\kappa_0}}}
{\eta_{\kappa_0}(q_{\kappa_0})}.
\]
Taking $a = 1$ in \eqref{eq:eta-kappa0-definition}, we obtain
\[
\eta_{\kappa_0}(q_{\kappa_0}) \geq \frac{\log Q_{\kappa_0}(q_{\kappa_0})}{\log q_{\kappa_0}}.
\]
Substitution proves $E_{\kappa_0} < N_{\kappa_0}^{\rm eff}$, and hence
\eqref{eq:explicit-NK} is an admissible cutoff.
\end{proof}

\begin{corollary}\label{cor:counterexample-escape}
Let $n_j \to \infty$ be any sequence satisfying
\[
\sigma_{-1}(n_j) > \RobinScale{n_j}.
\]
For each $j \in \bbZ_{>0}$, define
\[
\kappa_j^\dagger
\coloneqq
\inf\{\kappa > 1:\rkrho{\kappa}(n_j) \geq \RobinScale{n_j}\}.
\]
Then $\kappa_j^\dagger < \infty$ for each $j$, and
$\kappa_j^\dagger \to \infty$ as $j \to \infty$.
\end{corollary}

\begin{proof}
Finiteness follows from
$\rkrho{\kappa}(n_j) \nearrow \sigma_{-1}(n_j)$ as $\kappa \to \infty$.
If the thresholds failed to tend to infinity, some finite $\kappa_0$ would bound an
infinite subsequence.  By monotonicity and continuity in $\kappa$, each member
of that subsequence would satisfy
$\rkrho{\kappa_0}(n_j) \geq \RobinScale{n_j}$, contradicting
\Cref{thm:bounded-eventual} for sufficiently large $j$.
\end{proof}

The contrast with the original LCM-power family is therefore exact at the
level of tails.  There, every fixed $\kappa > 3 / 2$ retains an RH-equivalent
eventual inequality \cite[Theorem~1.5]{Fan-Mits-Molnar}.  Here every bounded
parameter range has only finitely many possible exceptions, unconditionally;
if RH is false, the reciprocal parameter at which classical counterexamples
become visible must itself escape to infinity.

\begin{remark}\label{rem:kappa-five-square-support}
The endpoint $\kappa = 5$ admits a shorter direct argument.  Put
$Q(n) \coloneqq \rkrho{5}(n)$ and $x \coloneqq 1 / p$.  Clearing denominators in
$Q(p^2) > Q(p)^2$ gives
\[
5\left((1 + x + x^2)^5 - (1 + x)^5\right)
- \left((1 + x)^5 - 1\right)^2
= x^6\left(4x^4 + 15x^3 + 30x^2 + 30x + 15\right) > 0.
\]
Thus exponent one cannot occur at a normalized $5$-endpoint.  If $P$ is its
largest prime factor, every prime $p \leq P$ occurs at least twice, so
\[
N \geq \primorial{P}^{\,2},
\qquad
Q(N) < \MertensProduct{P}.
\]
The explicit Mertens and Chebyshev estimates above, together with a short
finite check and \cite[Theorem~2]{MorrillPlatt} for the initial range, therefore
give an independent direct proof on $1 < \kappa \leq 5$.  The value $5$ is the
boundary of this exact square-support argument: for fixed $\kappa > 1$, as
$p \to \infty$ through the primes,
\[
\frac{\rkrho{\kappa}(p^2)}{\rkrho{\kappa}(p)^2}
= 1 + \frac{5 - \kappa}{12p^2} + O_\kappa(p^{-3}).
\]
For $\kappa > 5$ the square-support inequality therefore reverses for all
sufficiently large primes.  The general argument replaces exact square support
by \Cref{lem:effective-repeated-support}; the certified result
\Cref{thm:unconditional-interval} makes the special $\kappa = 5$ proof
unnecessary for the main theorem.
\end{remark}

\subsection{Certified extension to \texorpdfstring{$\kappa = \CertifiedKappa$}{the certified parameter frontier}}

\Cref{thm:bounded-eventual} leaves only a finite prefix at any specified parameter.
For \Cref{thm:unconditional-interval}, directed interval arithmetic resolves this
prefix at $\kappa = \CertifiedKappa$.  The same finite argument applies at any
larger parameter once the corresponding endpoint data have been independently
verified.

\begin{theorem}[{\Cref{thm:intro-unconditional}}]\label{thm:unconditional-interval}
For every real $1 < \kappa \leq \CertifiedKappa$ and every integer $n > 5040$, we have
\begin{equation}\label{eq:certified-robin}
\rkrho{\kappa}(n) < \RobinScale{n}.
\end{equation}
\end{theorem}

\begin{proof}
By parameter monotonicity, it suffices to prove the result at
$\kappa = \CertifiedKappa$.  Put
\[
Q_{\CertifiedKappa}(n) = \rkrho{\CertifiedKappa}(n),
\qquad
C_{\CertifiedKappa} = \theta(\CertifiedThetaArgument).
\]
We sharpen the argument proving \eqref{eq:bounded-tail-robin} at this fixed
parameter.  For $P \geq 58{,}271$, \eqref{eq:schoenfeld-lower} and
\eqref{eq:dusart-reciprocal-product} give the endpoint inequality as soon as
\[
\frac{1022}{3375}P > C_{\CertifiedKappa} + \frac13.
\]
Directed interval arithmetic gives
\[
\CertifiedThetaLower < C_{\CertifiedKappa} < \CertifiedThetaUpper,
\]
so every canonical normalized $\CertifiedKappa$-endpoint with
$P \geq \CertifiedAnalyticCutoff$ satisfies the Robin inequality.

For smaller $P$, a directed-interval computation evaluates the stronger
Mertens comparison from repeated support prime by prime.  It finds the last
failure at $P = \CertifiedLastBadPrime$ and verifies the comparison for every
$P \geq \CertifiedTailPrime$.  Thus every canonical endpoint with
$\CertifiedTailPrime \leq P < \CertifiedAnalyticCutoff$ also satisfies the Robin
inequality.

An independent directed-interval computation reconstructs the remaining
canonical normalized $\CertifiedKappa$-endpoints.  At
$P = \CertifiedTailPrime$ it first determines the required local maximizing
exponent.  For each smaller prime $p$, it constructs the upper concave hull of
the local points
\[
\left(a, \frac{\log Q_{\CertifiedKappa}(p^a)}{\log p}\right)
\]
and uses directed intervals to separate every transition slope that appears on
that hull.  Merging the resulting transitions produces
$\CertifiedEndpointEvents$ canonical endpoint events down to the entry of
$\CertifiedTailPrime$.  The reconstruction finds
$\CertifiedFirstEndpoint$ as the first endpoint above $5040$ and verifies a
positive Robin margin at every reconstructed endpoint above $5040$.  The
smallest certified logarithmic margin exceeds
\[
\CertifiedEndpointMargin.
\]

A final directed-interval scan checks the $\CertifiedInitialChecks$ integers
$5041 \leq n < \CertifiedFirstEndpoint$.  The quantity
$\RobinScale{n} - Q_{\CertifiedKappa}(n)$ attains its smallest value at
$n = \CertifiedDirectMinimumN$, where it exceeds
\[
\CertifiedDirectMargin.
\]
We have therefore proved the inequality through $\CertifiedFirstEndpoint$.
Above that point, the finite reconstruction handles consecutive endpoints
with largest prime factor below $\CertifiedTailPrime$, and
\Cref{prop:bounded-parameter-tail} handles all later endpoints.
\Cref{lem:endpoint-principle} then yields \eqref{eq:certified-robin} at
$\kappa = \CertifiedKappa$, and parameter monotonicity extends it to every
$1 < \kappa \leq \CertifiedKappa$.
\end{proof}

The final section gives the certification repository and independent
verification procedure.

The choice to optimize the normalized function is essential.  The
unnormalized reciprocal extremizers leave the term $-\omega(n)\log\kappa$ after
the common supporting inequality is applied.  This term is not determined by
$\log n$, so endpoint convexity cannot control the normalized Robin quotient.
\Cref{app:unnormalized-extremal} develops the unnormalized theory and gives a
concrete failure between consecutive events.

\appendix
\crefalias{section}{appendix}

\section{The unnormalized reciprocal extremal problem}\label{app:unnormalized-extremal}

The $\kappa$-colossally abundant numbers of \cite[\S4]{Fan-Mits-Molnar}
maximize $\ksigma{\kappa}(n)n^{-\kappa\epsilon}$.  The reciprocal analogue
replaces $\ksigma{\kappa}$ by $\rksigma{\kappa}$.  We record it here because
its extremizers do not support the endpoint transfer used in
\Cref{sec:unconditional-results}.

Fix $\kappa > 1$.

\begin{definition}\label{def:reciprocal-colossally-abundant}
For $\epsilon > -1 / \kappa$, an integer $N$ is
\defi{reciprocally $\kappa$-colossally abundant for parameter $\epsilon$}
if
\begin{equation}\label{eq:compact-reciprocal-ca-inequality}
\frac{\rksigma{\kappa}(N)}{N^{\kappa\epsilon}}
\ge
\frac{\rksigma{\kappa}(n)}{n^{\kappa\epsilon}}
\qquad\text{for }n \ge 1.
\end{equation}
\end{definition}

\subsection{Canonical unnormalized extremizers}

The defining quotient is multiplicative, so its optimization separates prime by
prime.  For a prime $p$ and an integer $a \ge 0$, put
\[
U_{p, \epsilon}(a)
\coloneqq
\rksigma{\kappa}(p^a)p^{-a\kappa\epsilon}.
\]
The restriction $\epsilon > -1 / \kappa$ is exactly what makes the local
maximization finite.

\begin{proposition}\label{prop:unnormalized-extremizers}
For every $\epsilon > -1 / \kappa$ and every prime $p$, the function
$U_{p, \epsilon}$ has a largest maximizing exponent $a_p$.  Only finitely many
$a_p$ are positive, and
\[
N_\kappa(\epsilon) \coloneqq \prod_p p^{a_p}
\]
is the largest reciprocally $\kappa$-colossally abundant integer for parameter
$\epsilon$.  The distinct values of $N_\kappa(\epsilon)$ form a locally finite
event sequence.  Moreover,
\[
N_\kappa(\epsilon) \longrightarrow \infty
\qquad\text{as }\epsilon \searrow -1 / \kappa.
\]
\end{proposition}

\begin{proof}
By the upper bound in \eqref{eq:discrete-mvt}, we have
\[
U_{p, \epsilon}(a)
\le
\kappa\left(\frac{p}{p - 1}\right)^{\kappa - 1}
p^{-a(1 + \kappa\epsilon)}
\qquad(a \ge 1).
\]
Since $1 + \kappa\epsilon > 0$, the right-hand side tends to $0$ as
$a \to \infty$.  A largest local maximizer therefore exists.  The same bound,
uniformly for $a \ge 1$, is less than $1 = U_{p, \epsilon}(0)$ for all sufficiently
large $p$.  Hence only finitely many $a_p$ are positive, and multiplicativity
gives the asserted global maximizer.

On a compact subinterval of $(-1 / \kappa, \infty)$, the exponent
$1 + \kappa\epsilon$ is bounded away from $0$.  The same estimate then bounds
uniformly both the primes and the exponents that can occur in a local
maximizer, proving local finiteness.

Finally,
\[
U_{p, -1 / \kappa}(1)
=
p\left((1 + p^{-1})^\kappa - 1\right) > 1
\]
by strict convexity of $t^\kappa$.  Thus every fixed prime occurs in
$N_\kappa(\epsilon)$ once $\epsilon > -1 / \kappa$ is sufficiently close to
$-1 / \kappa$, and the final assertion follows.
\end{proof}

\subsection{Unnormalized endpoints and the Robin quotient}

The event sequence in \Cref{prop:unnormalized-extremizers} is canonical for the
unnormalized quotient.  To transfer an inequality between consecutive events
as in \Cref{lem:endpoint-principle}, however, the factor left after applying a
common supporting parameter must be controlled by convexity in $\log n$.

\begin{remark}\label{rem:unnormalized-endpoint-transfer}
The Robin inequality concerns
\[
\frac{n\,\rksigma{\kappa}(n)}
{\kappa^{\omega(n)}\RobinScalePower{n}{\kappa - 1}}.
\]
For two consecutive unnormalized extremizers with a common supporting
parameter $\epsilon$, the logarithm of the remaining factor is
\begin{equation}\label{eq:unnormalized-endpoint-log-factor}
(1 + \kappa\epsilon)\log n
- (\kappa - 1)\log\!\bigl(\RobinScale{n}\bigr)
- \omega(n)\log\kappa.
\end{equation}
The first two terms in \eqref{eq:unnormalized-endpoint-log-factor} form a
smooth convex function of $\log n$; the last does not depend only on $\log n$.

At $\kappa = 2$, direct evaluation of the local factors in
\eqref{eq:compact-reciprocal-ca-inequality} gives consecutive unnormalized
extremizers $30$ and $210$.  But
\[
\rkRobinRatio{2}{30} \approx 0.7357686,
\qquad
\rkRobinRatio{2}{32} \approx 0.8822741,
\qquad
\rkRobinRatio{2}{210} \approx 0.5755899.
\]
Thus the interior integer $32$ has a larger normalized Robin ratio than either
endpoint.  The endpoint transfer in \Cref{lem:endpoint-principle} therefore
does not extend to these extremizers.
\end{remark}

\subsection{Local-gap comparisons}

One might instead try to bound the normalization gap prime by prime.  At
$\kappa = 2$, put
\begin{equation}\label{eq:k2-local-gap-factor}
r_{p, a} \coloneqq \frac{\rkrho{2}(p^a)}{\sigma_{-1}(p^a)}
= 1 - \frac{p - 1}{2(p^{a + 1} - 1)}.
\end{equation}
For an upper bound one would need to control $r_{p, a}$ from above.  But for
fixed $p$, it increases with $a$ and tends to $1$, so
\begin{equation}\label{eq:k2-local-gap-monotonicity}
a \ge 2\quad \Longrightarrow \quad r_{p, a} \ge r_{p, 2}.
\end{equation}
Replacing $a \ge 2$ by $2$ therefore gives a lower bound.

A second possible comparison uses the Euler product for squarefull integers.
An upper bound by the full Euler product would again require the opposite
inequality: since every factor lies in $(0, 1)$,
\begin{equation}\label{eq:squarefull-euler-product-comparison}
\prod_{p \mid n}(1 - p^{-3}) \geq \prod_p(1 - p^{-3}).
\end{equation}
Thus neither comparison gives the required upper bound.

\section*{Code and certification}

The code, numerical records, and certification files for the computer-assisted
arguments are available at
\url{https://github.com/grantmolnar/reciprocal-robin}, release~1.0.0, under the
MIT License.  The repository contains separate certificate generators and
independent verifiers, exact verification dependencies, and instructions for
reproducing the computations.
The full $\kappa = 100000$ endpoint certificate is $3{,}624{,}666{,}835$ bytes and
is not stored in Git.  The file \texttt{data/CERTIFICATION\_RESULTS.json}
records its SHA-256 digest and certified summary.  The released generator
reproduces the certificate, and the independent verifier checks it.

\section*{Acknowledgements}

We thank Steve Fan and Mits Kobayashi for discussions concerning
\cite{Fan-Mits-Molnar}, and Dave Platt for sharing the code used in
\cite{MorrillPlatt}.

\bibliographystyle{amsplain}

\begin{thebibliography}{10}

\bibitem{CaveneyNicolasSondow}
G.~Caveney, J.-L.~Nicolas, and J.~Sondow, \emph{On SA, CA, and GA numbers}, Ramanujan J. \textbf{29} (2012), 359--384.

\bibitem{Tenenbaum}
G.~Tenenbaum, \emph{Introduction to analytic and probabilistic number theory},
3rd ed., Graduate Studies in Mathematics, vol.~163, American Mathematical
Society, Providence, RI, 2015.

\bibitem{Dusart}
P.~Dusart, \emph{Explicit estimates of some functions over primes}, Ramanujan J. \textbf{45} (2018), no.~1, 227--251.

\bibitem{EssouabriSalinasToth}
D.~Essouabri, C.~Salinas Zavala, and L.~T\'{o}th, \emph{Mean values of multivariable multiplicative functions and applications to the average number of cyclic subgroups and multivariable averages associated with the LCM function}, J. Number Theory \textbf{236} (2022), 404--442.

\bibitem{Fan-Mits-Molnar}
S.~Fan, M.~Kobayashi, and G.~Molnar, \emph{A family of analogues to the Robin criterion}, Ramanujan J. \textbf{70} (2026), article~40.

\bibitem{Gronwall}
T.~H. Gr\"{o}nwall, \emph{Some asymptotic expressions in the theory of numbers}, Trans. Amer. Math. Soc. \textbf{14} (1913), no.~1, 113--122.

\bibitem{HeymanToth}
R.~Heyman and L.~T\'{o}th, \emph{Hyperbolic summation for functions of the GCD and LCM of several integers}, Ramanujan J. \textbf{62} (2023), no.~1, 273--290.

\bibitem{HilberdinkLucaToth}
T.~Hilberdink, F.~Luca, and L.~T\'{o}th, \emph{On certain sums concerning the gcd's and lcm's of $k$ positive integers}, Int. J. Number Theory \textbf{16} (2020), no.~1, 77--90.

\bibitem{HilberdinkToth}
T.~Hilberdink and L.~T\'{o}th, \emph{On the average value of the least common multiple of $k$ positive integers}, J. Number Theory \textbf{169} (2016), 327--341.

\bibitem{KimReciprocalLCM}
S.~Kim, \emph{On the reciprocal sum of lcm of $k$-tuples}, Res. Number Theory \textbf{8} (2022), article~44.

\bibitem{MorrillPlatt}
T.~Morrill and {D.\,J.} Platt, \emph{Robin's inequality for 20-free integers}, Integers \textbf{21} (2021), \#A28.

\bibitem{Robin}
G.~Robin, \emph{Grandes valeurs de la fonction somme des diviseurs et hypoth\`{e}se de Riemann}, J. Math. Pures Appl. (9) \textbf{63} (1984), no.~2, 187--213.

\bibitem{Schoenfeld}
L.~Schoenfeld, \emph{Sharper bounds for the Chebyshev functions $\theta(x)$ and $\psi(x)$. II}, Math. Comp. \textbf{30} (1976), no.~134, 337--360.

\bibitem{TothSurvey}
L.~T\'{o}th, \emph{Multiplicative arithmetic functions of several variables: a survey}, in \emph{Mathematics Without Boundaries: Surveys in Pure Mathematics}, T.~M. Rassias and P.~M. Pardalos (eds.), Springer, New York, 2014, 483--514.

\end{thebibliography}

\end{document}